\documentclass[UTF8,fontset = windows,preprint, review,12pt]{elsarticle}

\usepackage{ctex}
\usepackage{amssymb}
\usepackage{graphicx}

\usepackage{epstopdf}

\usepackage{subfigure}
\usepackage{extarrows}
\usepackage{fancyhdr}
\usepackage{amsthm}
\usepackage{extarrows}
\usepackage{CJK}
\usepackage{caption}
\usepackage{float}
\usepackage{algorithm}
\usepackage{algorithmic}
\usepackage{multirow}
\usepackage{amsmath}
\usepackage[numbers]{natbib}
\usepackage{stmaryrd}

\usepackage{color}
\usepackage{hyperref}

\usepackage[toc,page,title,titletoc,header]{appendix}
\usepackage{appendix}
\theoremstyle{definition}

\theoremstyle{definition}

\theoremstyle{definition}
\usepackage{latexsym, bm}
\usepackage{amsmath}
\usepackage{amssymb}
\usepackage{amsthm}
\usepackage{fancyhdr}
\usepackage{mathrsfs}
\usepackage{wasysym}
\usepackage{float}
\usepackage{graphicx}
\usepackage{multirow}
\usepackage{titlesec}
\usepackage{pgfplots}
\usepackage{tikz}
\usepackage{amsmath}
\usepackage{amssymb}
\usepackage{amsfonts}
\usepackage{amsthm}
\usepackage{fancyhdr}
\usepackage{latexsym,bm}
\usepackage{mathrsfs}
\usepackage{wasysym}
\usepackage{graphicx}
\usepackage{subfigure}
\usepackage{ctex}
\biboptions{numbers,sort&compress}

\usetikzlibrary{arrows,shapes,positioning}
\usetikzlibrary{decorations.markings}
\tikzstyle arrowstyle=[scale=1]
\tikzstyle directed=[postaction={decorate,decoration={markings,
    mark=at position .65 with {\arrow[arrowstyle]{stealth}}}}]
\tikzstyle reverse directed=[postaction={decorate,decoration={markings,
    mark=at position .65 with {\arrowreversed[arrowstyle]{stealth};}}}]

\newtheorem{thrm}{Theorem}[section]

\newtheorem{lemm}[thrm]{Lemma}
\newtheorem{re}[thrm]{Remark}

\allowdisplaybreaks

\begin{document}

\begin{frontmatter}

\title{Multirate iterative scheme with multiphysics finite element method for a fluid-saturated poroelasticity\tnoteref{1}}\tnotetext[1]{The work is supported by the National Natural Science Foundation of China under grant No.11971150, the major projects of international science and technology cooperation Henan University under grant No. 2021ybxm07 and the cultivation project of first class subject of Henan University under grant No. 2019YLZDJL08.\\
*Corresponding author. Email: zhihaoge@henu.edu.cn }
\author{Zhihao G${\rm e^{1,2 *}}$,\ Xiangzi F${\rm u^{1}}$}
\address{$ ^1$ School of Mathematics and Statistics, Henan University, Kaifeng 475004, People's Republic of China\\
$^2$ Henan Engineering Research Center for Artificial Intelligence Theory and Algorithms, Henan University, Kaifeng 475004, People's Republic of China}


\begin{abstract}
In this paper, we propose a multirate iterative scheme with multiphysics finite element method for a fluid-saturated poroelasticity model. Firstly, we reformulate the original model into a fluid coupled problem to apply the multiphysics finite element method for the discretization of the space variables, and we design a multirate iterative scheme on the time scale which solve a generalized Stokes problem in the coarse time size and solve the diffusion problem in the finer time size according to the characteristics of the poroelasticity problem. Secondly, we prove that the multirate iterative scheme is stable and the numerical solution satisfies some energy conservation laws, which are important to ensure the uniqueness of solution to the decoupled computing problem. Also, we analyze the error estimates to prove that the proposed numerical method doesn't reduce the precision of numerical solution and greatly reduces the computational cost. Finally, we give the numerical tests to verify the theoretical results and draw a conclusion to summary the main results in this paper.
\end{abstract}
\begin{keyword}
Poroelasticity, multiphysics finite element methods, multirate iterative scheme.
\end{keyword}
\end{frontmatter}

\thispagestyle{empty}

\numberwithin{equation}{section}
\section{Introduction}
Poromechanic is a branch of continuum mechanics and acoustics, which is a fluid-solid interaction system at pore scale. If the solid is an elastic material, then the subject of the study is known as poroelasticity. In this paper, we study the behavior of a fluid-saturated poroelasticity model as follows:
\begin{eqnarray}
-{\rm div}\sigma({\bf u})+\alpha\nabla p={\bf f} & \mbox{in }  \Omega_{T}:=\Omega\times (0,T)\subset \mathbb{R}^{d}\times (0,T),\label{eq-6}\\
(c_{0}p+\alpha {\rm div} {\bf u})_{t} + {\rm div} {\bf v}_{f} =\phi & \mbox{in }  \Omega_{T}.\label{eq-7}
\end{eqnarray}
Here $\sigma(\bf{u})$ is called the (effective) stress tensor, and defined by 
\begin{eqnarray}
\sigma(\mathbf{u}):~=\mu \varepsilon(\mathbf{u})+\lambda {\rm {\rm div}}\mathbf{u}\textit{I}, 
\end{eqnarray}
where $\varepsilon(\mathbf{u})=\frac{1}{2}(\nabla \mathbf{u}+\nabla\mathbf{u}^T)$ and ${\bf v}_{f}$ is the volumetric solvent flux and is the well-known Darcy's law and given by
\begin{eqnarray}
\mathbf{v}_f:~=-\frac{K}{\mu_f}(\nabla p-\rho_f \mathbf{g}). \label{eq-9}
\end{eqnarray}
In addition, $\Omega \subset \mathbb{R}^{d}(d=2,3)$ is a bounded polygonal domain with the boundary $\partial\Omega$,  ${\bf u}$ denotes the displacement vector of the solid and\ $p$ denotes the pressure of the solvent, $\bf{f}$ is the body force, $I$ denotes the\ $d\times d$ identity matrix and $\varepsilon({\bf u})$ is known as the strain tensor. The parameters in the above model list as follows: Lam\'e constants\ $\lambda$ and \ $\mu$; the permeability tensor\ $K=K(x)$ which is assumed to be symmetric and uniformly positive definite in the sense that there exist positive constants\ $K_{1}$ and \ $K_{2}$ such that\ $K_{1}|\zeta|^{2}\leq K(x)\zeta\cdot \zeta \leq K_{2} |\zeta|^{2}$ for a.e.  $x\in\Omega$ and any\ $\zeta\in \mathbb{R}^{d}$; the solvent viscosity\ $\mu_{f}$, Biot-Willis constant $\alpha$, and the constrained specific storage coefficient $c_{0}$. Also, we denote $\hat{\sigma}({\bf u},p):=\sigma({\bf u})-\alpha p I$ by the total stress tensor.

To close the problem (\ref{eq-6})-(\ref{eq-7}), we impose the following boundary and initial conditions:
\begin{eqnarray}
\widehat{\sigma}(\mathbf{u},p)\mathbf{n}=\sigma(\mathbf{u})\mathbf{n}-\alpha pI\mathbf{n}=\mathbf{f}_1& \quad \mathrm{on}\ & \partial\Omega_T:~=\partial\Omega\times(0,T),\\ \label{eq-10}
\mathbf{v}_f\cdot\mathbf{n}=-\frac{K}{\mu_f}(\nabla p-\rho_f \mathbf{g})\cdot \mathbf{n}=-\phi_1& \quad \mathrm{on}\ & \partial\Omega_T,\\ \label{eq-11}
\mathbf{u}=\mathbf{u_0},~~~p=p_0& \quad \mathrm{in}\ & \Omega\times\{t=0\}.\label{eq-12}
\end{eqnarray}

We remark that the Lam\'e constant\ $\mu$ is also called the  shear modulus and denoted by\ $G$, and\ $B:=\lambda+\frac{2}{3} G$ is called the  bulk modulus. \ $\lambda, \mu$ and\ $B$ are computed from the  Young's modulus\ $E$ and the Poisson ratio\ $\nu$ by the following formulas:
\[
\lambda=\frac{E\nu}{(1+\nu)(1-2\nu)},\qquad \mu=G=\frac{E}{2(1+\nu)}, \qquad
B=\frac{E}{3(1-2\nu)}.
\]

The problem (\ref{eq-6})-(\ref{eq-12}) is widely applied  to many fields such as biomedical and chemical systems, environmental and reservoir engineering, for the details, one can refer to  \cite{M.Doi2010,T.Tanaka1979,K.Terzaghi1943,H.Byrne2003, C.C.Swan2003,coussy04,J.Rutqvist2003,R.Rajapakse1993,G.A.Behie2001, P.J.Phillips200711(2):131-144} and the  references therein.

Since the problem (\ref{eq-6})-(\ref{eq-12}) and its solution domain are very complicated, it is difficult to find the analytical solution of the problem. Therefore, the finite element method are usually applied to approximate the poroelastic problem. In recent years, many scholars have proposed different finite element methods to study the poroelasticity model, the main difficulty is ``locking'' phenomenon, for the details, one can see \cite{P.J.Phillips200913(1):5-12}.  The authors of \cite{P.J.Phillips200711(2):145-158, P.J.Phillips200711(2):131-144} proposed and analyzed a semi-discrete and a fully discrete mixed finite element method which simultaneously approximate the pressure and its gradient along with the displacement vector field, based on the same or similar idea, some stabilized finite element methods are designed, one can see \cite{Hu2017,JJLee2016,JJLee2018} and the therein references. To prevent ``locking'' phenomenon and reveal the multi physical processes, the authors of \cite{X.B.Feng2010,GeGuan2018} proposed fully discrete finite element method to describe the expansion dynamics of polymer gels under mechanical constraints. Later, the authors of \cite{X.B.Feng2014} proposed multiphysics finite element method for poroelasticity model, the key idea of multiphysics finite element method is to reconstruct the original poroelastic problem by introducing two pseudo-pressure fields, which makes the original problem is decoupled into two sub-problems at each time step. For the coupled fluid problem, the generalized Stokes problem changes slow and the diffusion problem changes fast along with time, in order to study this phenomenon, the authors of \cite{Kumar2} propose a multirate iterative scheme based on a mixed formulation. Ge and Ma in \cite{Z.H.Ge2018} proposed a multirate iterative scheme based on the multiphysics discontinuous Galerkin method and gave the optimal convergent order estimates, which is very different from the method of \cite{Kumar2}. However, the multiphysics discontinuous Galerkin method has a large amount of computation. In this paper, in order to reduce the computational while retaining the precision, we use the multiphysics finite element method for the discretization of the space variables and adopt a multirate iterative scheme on the time scale which solve a generalized Stokes problem in the coarse time size and solve the diffusion problem in the finer time size. And we prove that the multirate iterative scheme is stable and the numerical solution satisfies some energy conservation laws. Also, we prove that it doesn't reduce the precision of numerical solution and compared with the  the multiphysics discontinuous  Galerkin method, the multirate iterative scheme based on the multiphysics finite element method greatly reduces computation time.

The remainder of this paper is organized as follows. In Section \ref{sec2}, we present the reconstruction of poroelasticity model and the preliminary knowledge which is needed to study the poroelasticity problem. In Section \ref{sec3}, we propose and analyze the multirate iterative scheme with the multiphysics finite element method for the reformulated model, and give the optimal error estimates. In Section \ref{sec4}, we give some numerical examples to verify theoretical results. Finally, we draw a conclusion to summary the main results in this paper.

\numberwithin{equation}{section}
\section{Multiphysics reformulation of poroelasticity model}\label{sec2}
To reveal the multi physical processes and propose an effective numerical method, we introduce the new variable $q={\rm div} {\bf u}$ and denote $\eta=c_{0}p+\alpha q, \xi=\alpha p -\lambda q$, it is easy to check that
\begin{equation}
p=\kappa_{1} \xi + \kappa_{2} \eta,  q=\kappa_{1} \eta-\kappa_{3} \xi,\label{eq-13}
\end{equation}
where
\begin{equation}\label{eq-14}
\kappa_{1}= \frac{\alpha}{\alpha^{2}+\lambda c_{0}},
\quad \kappa_{2}=\frac{\lambda}{\alpha^{2}+\lambda c_{0}}, \quad
\kappa_{3}=\frac{c_{0}}{\alpha^{2}+\lambda c_{0}}.
\end{equation}

Using (\ref{eq-13}), we reformulate  (\ref{eq-6})-(\ref{eq-9}) into the following system:
\begin{eqnarray}
-\mu {\rm div}\varepsilon(\mathbf{u})+\nabla\xi=\mathbf{f} &&\quad \mathrm{in}\ \Omega_T, \label{eq-15}\\
\kappa_3\xi+{\rm {\rm div}}\mathbf{u}=k_1\eta &&\quad \mathrm{in}\ \Omega_T, \label{eq-16}\\
\eta_t- \frac{1}{\mu_f}{\rm {\rm div}}[K(\nabla(k_1\xi+k_2\eta)-\rho_f \mathbf{g})]=\phi &&\quad \mathrm{in}\ \Omega_T. \label{eq-17}
\end{eqnarray}
The boundary and initial conditions (\ref{eq-10})-(\ref{eq-12}) can be rewritten as
\begin{eqnarray}
\sigma(\mathbf{u})\mathbf{n}-\alpha (\kappa_1\xi+\kappa_2\eta)I\mathbf{n}=\mathbf{f}_1&& \quad \mathrm{on}\ \partial\Omega_T:=\partial\Omega\times (0,T),\label{eq-18}\\
-\frac{K}{\mu_f}(\nabla(\kappa_1\xi+\kappa_2\eta)-\rho_f \mathbf{g})\cdot \mathbf{n}=\phi_1&& \quad \mathrm{on}\  \partial\Omega_T,\label{eq-19}\\
\mathbf{u}=\mathbf{u_0},~~~p=p_0&& \quad \mathrm{in}\  \Omega\times\{t=0\}.\label{eq-20}
\end{eqnarray}
\begin{re}\label{rem200131-1}
From the problem (\ref{eq-15})-(\ref{eq-20}), we know that  $(\mathbf{u}, \xi)$ satisfies the generalized Stokes problem of the displacement vector field along with the pseudo-pressure field, and $\eta$ satisfies the diffusion problem of the pseudo-pressure field. Thus, the problem (\ref{eq-15})-(\ref{eq-20}) reveals the multiphysics deformation and diffusion process.
\end{re}
\begin{re}\label{rem200131-2}
In the problem (\ref{eq-15})-(\ref{eq-20}), the original variable $p$ is no longer the primary variable, but it is calculated as a linear combination of\ $\xi$ and\ $\eta$, which  can be  updated by (\ref{eq-13}), just because of this, it is possible to design an effective numerical method without ``locking'' phenomenon for the pressure of $p$.
\end{re}

Next, we introduce some function spaces. For any Banach space $B$, we let ${\bf B}=[B]^{d}$, and denote $\bf{B}'$ by its dual space,
and $\|\cdot\|_{L^{p}(B)}$ is a shorthand notation for $\|\cdot\|_{L^{p}((0, T); B)}$.  And we denote  $(\cdot,\cdot)$ and $\langle \cdot,\cdot\rangle$ by the standard $L^2(\Omega)$ and $L^2(\partial\Omega)$ inner products, respectively. Also, we need to introduce the function spaces:
\begin{align*}
&L^{2}_{0}(\Omega):=\{q\in L^{2}(\Omega); (q,1)=0\}, {\bf X}:= {\bf H}^{1}(\Omega).
\end{align*}
Denote $
{\bf RM}:=\{{\bf r}={\bf a}+{\bf b}\times x;{\bf a}, {\bf b}, x\in \mathbb{R}^{d}\}$
by the space of infinitesimal rigid motions. From  \cite{S.C.Brenner2008}, it is well known that\ $\bf{RM}$ is the kernel of the strain operator $\varepsilon$, that is,\ $\bf{r} \in \bf{RM}$ if and only if\ $\varepsilon(\bf{r})=0$. Hence, we have
\begin{equation}\label{eq-2}
\varepsilon({\bf r})={\bf 0}, {\rm div} {\bf r}=0 \qquad\forall \bf{r} \in \bf{RM.}
\end{equation}
Let ${\bf L}_{\bot}^{2}(\partial\Omega)$ and ${\bf H}_{\bot}^{1}(\Omega)$ denote respectively the subspaces of ${\bf L}^{2}(\partial\Omega)$ and ${\bf H}^{1}(\Omega)$ which are orthogonal to ${\bf RM}$, that is,
\begin{align*}
&{\bf H}^{1}_{\bot}(\Omega):=\{{\bf v}\in {\bf H}^{1}(\Omega);({\bf v},{\bf r})=0\qquad\forall {\bf r}\in {\bf RM}\},\\
&{\bf L}_{\bot}^{2}(\partial\Omega):=\{{\bf g}\in{\bf L}^{2}(\partial\Omega);\langle {\bf g},{\bf r}\rangle=0\qquad
\forall {\bf r}\in {\bf RM}\}.
\end{align*}

As for the existence and uniqueness of weak solution of the problem (\ref{eq-15})-(\ref{eq-20}), one can refer to \cite{X.B.Feng2014}, here we omit the details.

\numberwithin{equation}{section}
\section{Multirate iterative scheme with multiphysics finite element method}\label{sec3}

In this section, we propose the finite element solution\ $(\mathbf{u}_{h},\xi_{h},\eta_{h })$ of the problem ~(\ref{eq-15})-(\ref{eq-20}). We use the multiphysics finite element method for the discretization of the space variables and adopt a multirate iterative scheme on the time scale, which solve a generalized Stokes problem in the coarse time size and solve the diffusion problem in the finer time size.

\subsection{Multirate iterative scheme}

Let $\mathcal{T}_{h}$ be a quasi-uniform triangulation or rectangular partition of\ $\Omega$ with mesh size\ $h$, and\ $\overline{\Omega}=\bigcup_{K\in\mathcal{T}_{h}}\overline{K}$. The time interval $[0,T]$ is divided into $N$ equal intervals, denoted by $[t_{n-1}, t_{n}], n=1,2,...N$,  and $\Delta t=\frac{T}{N}$, then $t_n=n\Delta t$.


In this paper, we use the following Taylor-Hood element:
\begin{eqnarray*}
{\bf X}_{h} =\{{\bf v}_{h}\in {\bf C}^{0}(\overline{\Omega});{\bf v}_{h}|_{K}\in {\bf P}_{2}(K)~~\forall K\in \mathcal{T}_{h} \},\\
M_{h} =\{\varphi_{h}\in C^{0}(\overline{\Omega}); \varphi_{h}|_{K}\in P_{1}(K)
~~\forall K\in \mathcal{T}_{h} \}.
\end{eqnarray*}

Finite element approximation space\ $W_{h}$ for\ $\eta$ variable can be chosen
independently, any piecewise polynomial space is acceptable provided that
$W_{h} \supset M_{h}$. The most convenient choice is\ $W_{h} =M_{h}$.

Recall that $\bf{RM}$ denotes the space of the infinitesimal rigid motions, evidently,\ ${\bf RM}\subset {\bf X}_h$. We define
\begin{equation}\label{eq-22}
{\bf V}_{h}=\{{\bf v}_{h}\in {\bf X}_h;({\bf v}_{h},{\bf r})=0,\forall {\bf r}\in {\bf RM} \}.
\end{equation}

It is easy to check that ${\bf X}_{h}={\bf V}_{h}\bigoplus {\bf RM}$. From \cite{X.B.Feng2010}, we know that there holds the following of inf-sup condition:
\begin{equation}\label{eq-23}
\sup_{{\bf v}_{h}\in{\bf V}_{h}}\frac{({\rm div}{\bf v}_{h},\varphi_{h})}{{\|{\bf v}_{h}\|}_{H^{1}(\Omega)}}
\geq \beta_{1} \|\varphi_{h}\|_{L^{2}(\Omega)} \quad \forall \varphi_{h}\in M_{0h}, \quad \beta_{1}>0.
\end{equation}

Next, we
define the bilinear forms as follows:
\begin{equation}\label{eq-24}
a({\bf u},{\bf v})= \mu(\varepsilon({\bf u}),\varepsilon({\bf v})),
\end{equation}
\begin{equation}\label{eq-25}
b({\bf v},\xi)=-({\rm div} {\bf v},\xi),
\end{equation}
\begin{equation}\label{eq-26}
c(\xi,\varphi)=k_{3}(\xi,\varphi).
\end{equation}

%
%

Now, we propose a multirate iterative scheme with multiphysics finite element method for the problem (\ref{eq-15})-(\ref{eq-20}) as follows:\\
(\romannumeral1) Compute ~$\mathbf{u}_{h}^{0}\in {\bf V}_{h}$ and ~ $q_{h}^{0}\in M_{h}$ by
\begin{eqnarray*}
&&\mathbf{u}_h^0=\mathcal{R}_h\mathbf{u}^0, ~~~p_h^0=\mathcal{Q}_h p_0, ~~~ q_h^0=\mathcal{Q}_h q_0~ (q_0=\mathrm{{\rm div}}\mathbf{u}_0),\\
&&\eta_h^0= c_0 p_h^0+\alpha q_h^0, ~~~ \xi_h^0=\alpha p_h^0-\lambda q_h^0,
\end{eqnarray*}
where $\mathcal{R}_h$ and $\mathcal{Q}_h$ are defined by (\ref{eq-47}) and (\ref{eq-50}), respectively.\\
(\romannumeral2) For\ $n=0, \,1, \,2, \dots$, do the following three steps:

\emph{Step~1:} Solve for\ $({\textbf{u}_h^{(n+1)m}} ,{\xi_h^{(n+1)m}}) \in {\bf V}_{h}\times M_{h}$ such that
\begin{eqnarray}
&&  a({\textbf{u}_{h}^{(n+1)m}},{\textbf{v}_{h}})
+b({\textbf{v}_{h}},{\xi_{h}^{(n+1)m}})
 =({\textbf{f}},{\textbf{v}_{h}})
 +\langle{\textbf{f}_{1}},{\textbf{v}_h}\rangle, ~\forall{\textbf{v}_{h}}\in{{\bf V}_{h}},\label{eq-28}\\
 &&-b({\textbf{u}_{h}^{(n+1)m}},{\varphi_{h}})
 +c({\xi_{h}^{(n+1)m}},{\varphi_{h}})=
k_{1}({\eta_{h}^{(n+\theta)m}},{\varphi_{h}}), \forall \varphi_{h}\in M_{h}.\label{eq-29}
\end{eqnarray}

\emph{Step 2:} For\ $k=1, \,2, \dots, m,$ solve for\ $\eta_{h}^{nm+k}\in M_{h}$ such that
\begin{eqnarray}
&&({d_{t}\eta_{h}^{nm+k}},\psi_{h})+\frac{1}{\mu_{f}}({K(\nabla(k_{1}\xi_{h}^{(n+1)m} +k_{2}\eta_{h}^{nm+k})-{\rho_{f}}\mathbf{g})},{\nabla\psi_{h}})\nonumber\\
&&=(\phi,\psi_{h})+\langle{\phi_{1}},\psi_{h}\rangle.\label{eq-30}
\end{eqnarray}

\emph{Step 3:} Update $p_{h}^{(n+1)m}$ and $q_{h}^{(n+1)m}$ by
\begin{equation}\label{eq-31}
p_{h}^{(n+1)m}=k_{1}\xi_{h}^{(n+1)m}+k_{2}\eta_{h}^{(n+\theta)m},~~ q_{h}^{(n+1)m}=k_{1}\eta_{h}^{(n+1)m}-k_{3}\xi_{h}^{(n+1)m},
\end{equation}
where $\theta=0$ or $1$, $m$ is a positive integer, $d_{t}\eta_{h}^{nm+k}=\frac{\eta_{h}^{nm+k}-\eta_{h}^{nm+k-1}}{\Delta t}$.

\subsection{Stability analysis}

\begin{lemm}\label{lma1912-1}
Let\ $\{({\textbf{u}_{h}^{(n+1)m}}, {\xi_{h}^{(n+1)m}}, {\eta_{h}^{nm+k}})\}_{n\geq0, 1\leq k \leq m}$ be the solution of the multirate iterative scheme (\ref{eq-28})-(\ref{eq-31}), then we have
\begin{equation}\label{eq-40}
J_{h,\theta}^{l+1}+S_{h,\theta}^{l}=J_{h,\theta}^{0}
 ~~~~~{\rm for}~~l\geq0,~~\theta=0,~1,
\end{equation}
where
\begin{eqnarray}\begin{aligned}\nonumber
&J_{h,\theta}^{l+1}:=\frac{1}{2}[\mu \|\varepsilon(\mathbf{u}_{h}^{(l+1)m})\|_{L^{2}(\Omega)}^{2}
+\kappa_{2}\|\eta_{h}^{(l+\theta)m}\|_{L^{2}(\Omega)}^{2}
+\kappa_{3}\|\xi_{h}^{(l+1)m}\|_{L^{2}(\Omega)}^{2}\\
&-2({\mathbf{f}},{\mathbf{u}_{h}^{(l+1)m}})-2\langle{\mathbf{f}_{1}},
{\mathbf{u}_{h}^{(l+1)m}}\rangle],\\
&S_{h,\theta}^{l}:=\Delta t \sum\limits_{n=0}^{l}\sum_{k=1}^{m}[\frac{K}{\mu_{f}}(\nabla p_{h}^{nm+k}-\rho_{f}\mathbf{g}, \nabla p_{h}^{nm+k})+\frac{\kappa_{2} \Delta t}{2}\|d_{t}\eta_{h}^{(n-1+\theta)m+k}\|_{L^{2}(\Omega)}^{2}\\
&-(\phi,p_{h}^{nm+k})-\langle{\phi_{1}},{p_{h}^{nm+k}\rangle}
-(1-\theta)\frac{\kappa_{1}  k \Delta t}{\mu_{f}}(K d_{t}^{(k)} \nabla \xi_{h}^{nm+k},{\nabla p_{h}^{nm+k}})_{L^{2}(\Omega)}] \\
&+m\Delta t \sum\limits_{n=0}^{l}( \frac{\mu m\Delta t}{2}\|d_{t}^{(m)}\varepsilon(\mathbf{u}_{h}^{(n+1)m})\|_{L^{2}(\Omega)}^{2}+\frac{\kappa_{3} m\Delta t}{2}\|d_{t}^{(m)} \xi_{h}^{(n+1)m}\|_{L^{2}(\Omega)}^{2}),
\end{aligned}\end{eqnarray}
where $d_{t}^{(m)}\eta_{h}^{nm}$ is defined by
\begin{eqnarray} d_{t}^{(m)}\eta_{h}^{nm}=\frac{\eta_{h}^{nm}-\eta_{h}^{nm-m}}{m \Delta t}.\label{eq200121-1}
\end{eqnarray}
\end{lemm}
\begin{proof}
Since the proof for the case of $\theta=1$ is simple, so we we only consider the case of $\theta=0$.
Setting $\textbf{v}_{h}=d_{t}^{(m)}\textbf{u}_h^{(n+1)m}$ in (\ref{eq-28}), $\varphi_{h}=\xi_{h}^{(n+1)m}$ in (\ref{eq-29}), and $\psi_{h}=p_{h}^{nm+k}$ in (\ref{eq-30}) after lowing the degree from $n$ to $n-1$ and then summing over $k$ from $1$ to $m$, we get
\begin{eqnarray}
&&\mu({\varepsilon(\textbf{u}_{h}^{(n+1)m})}, {\varepsilon(d_{t}^{(m)}\textbf{u}_{h}^{(n+1)m})})-({\xi_{h}^{(n+1)m}}, {\mathrm{{\rm div}}d_{t}^{(m)}\textbf{u}_{h}^{(n+1)m}})\nonumber\\ &&=({\textbf{f}}, d_{t}^{(m)}{\textbf{u}_{h}^{(n+1)m}})+\langle{\textbf{f}_{1}}, d_{t}^{(m)}{\textbf{u}_{h}^{(n+1)m}}\rangle,\label{eq-41}\\
&&\kappa_{3}({d_{t}^{(m)}\xi_{h}^{(n+1)m}},{\xi_{h}^{(n+1)m}})
 +({\mathrm{{\rm div}}d_{t}^{(m)}\textbf{u}_{h}^{(n+1)m}}, {\xi_{h}^{(n+1)m}})\nonumber\\
&&=\kappa_{1}({d_{t}^{(m)}\eta_{h}^{nm}},{\xi_{h}^{(n+1)m}}),\label{eq-42}\\
&&\sum_{k=1}^{m}[(d_{t}\eta_{h}^{(n-1)m+k},p_{h}^{nm+k})+\frac{1}{\mu_{f}}({K(\nabla(\kappa_{1}\xi_{h}^{nm}+\kappa_{2}\eta_{h}^{(n-1)m+k})-{\rho_{f}}\mathbf{g})},{\nabla p_{h}^{nm+k}})\nonumber\\
&&=\sum_{k=1}^{m} [ (\phi, {p_{h}^{nm+k}})+\langle{\phi_{1}}, {p_{h}^{nm+k}}\rangle ].\label{eq-43}
\end{eqnarray}

The first term on the left-hand side of (\ref{eq-43}) can be rewritten as
\begin{eqnarray}\label{eq-44}
&&\sum_{k=1}^{m}({d_{t}\eta_{h}^{(n-1)m+k}},p_{h}^{nm+k})
=\sum_{k=1}^{m}({d_{t}\eta_{h}^{(n-1)m+k}},{\kappa_{1}\xi_{h}^{nm+k}
+\kappa_{2}\eta_{h}^{(n-1)m+k}})\nonumber\\
&&=\sum_{k=1}^{m}(\frac{\kappa_{2}\Delta t}{2}\|d_{t}\eta_{h}^{(n-1)m+k}\|_{L^{2}(\Omega)}^{2}+\frac{\kappa_{2}}{2}d_{t}\|\eta_{h}^{(n-1)m+k}\|_{L^{2}(\Omega)}^{2})\nonumber\\
&&+\kappa_{1}m({d_{t}^{(m)}\eta_{h}^{nm}},{\xi_{h}^{(n+1)m}} ).
\end{eqnarray}

Moreover, we have
\begin{eqnarray}
&&\sum_{k=1}^{m}\frac{K}{\mu_{f}}({\nabla (k_{1}\xi_{h}^{nm}+k_{2}\eta_{h}^{(n-1)m+k})-\rho_{f} \mathbf{g}},\nabla p_{h}^{nm+k})\nonumber\\
&&=\sum_{k=1}^{m}[\frac{K}{\mu_{f}}({\nabla p_{h}^{nm+k}-\rho_{f} \mathbf{g}},\nabla p_{h}^{nm+k})-\frac{k_{1} K k\Delta t}{\mu_{f}}({d_{t}^{(k)}\nabla\xi_{h}^{nm+k}}, {\nabla p_{h}^{nm+k}})],\label{eq-45}\\
&&k_{3}({d_{t}^{(m)}\xi_{h}^{(n+1)m}},{\xi_{h}^{(n+1)m}})
=\frac{k_{3}}{2}d_{t}^{(m)}\|\xi_{h}^{(n+1)m}\|_{L^{2}(\Omega)}^{2}+\frac{k_{3} m\Delta t}{2}\|d_{t}^{(m)} \xi_{h}^{(n+1)m}\|_{L^{2}(\Omega)}^{2}.\label{eq-46}
\end{eqnarray}

Adding (\ref{eq-41})-(\ref{eq-42})(after appling the summation operator $m\Delta t\sum\limits_{n=0}^{l}$) and (\ref{eq-43})(after applying the summation operator $\Delta t\sum\limits_{n=0}^{l}$), using (\ref{eq-44})-(\ref{eq-46}), we see that (\ref{eq-40}) holds. The proof is complete.
\end{proof}

\begin{lemm}\label{lma1912-2}
Let $\{({\textbf{u}_{h}^{(n+1)m}}, {\xi_{h}^{(n+1)m}}, {\eta_{h}^{nm+k}})\}_{n\geq0, 1\leq k \leq m}$ be the solution of the multirate iterative scheme (\ref{eq-28})-(\ref{eq-31}), then there hold
\begin{eqnarray}
&&(\eta_{h}^{nm},1)=C_{\eta}(t_{nm})~~~{\rm for}~~ n=0, \,1, \,2 \cdots,\label{eq-33}\\
&&(\xi_{h}^{nm},1)=C_{\xi} (t_{(n-1+\theta)m})~~~~ {\rm for}~~n=1-\theta, \,1, \,2, \cdots,\label{eq-34}\\
&&\langle{\textbf{u}_{h}^{nm} \cdot {\textbf{n}} },1\rangle=C_{\textbf{u}}(t_{(n-1+\theta)m})~~~~ {\rm for}~~n=1-\theta, \,1, \,2, \cdots,\label{eq-35}
\end{eqnarray}
where $C_{\eta} (t)=(\eta(\cdot, t), 1)$, $C_{\xi} (t)=(\xi(\cdot, t), 1)=\frac{1}{d+\mu \kappa_{3}}(\mu k_{1}C_{\eta} (t)-(\mathbf{f}, x)-\langle\mathbf{f}_1, x\rangle)$, $C_{\textbf{u}}(t)=\langle\textbf{u}(\cdot, t)\cdot{\textbf{n}}, 1\rangle$.
\end{lemm}
\begin{proof}
Taking $\psi_{h}=1$ in (\ref{eq-30}) and summing over $k$ from $1$ to $m$ and over $n$ from $0$ to $l$, we get
\begin{equation}\label{eq-36}
({\eta_{h}^{(l+1)m}},1)=(\eta_h^0,1)
+[(\phi,1)+\langle{\phi_1,1}\rangle]t_{(l+1)m}=C_{\eta}(t_{(l+1)m}),\ l=0, \,1, \,2, \cdots.
\end{equation}
From (\ref{eq-36}), we see that (\ref{eq-33}) holds.

Taking $\textbf{v}_{h}=x$ in (\ref{eq-28}) and $\varphi_{h}=1$ in (\ref{eq-29}), we have
\begin{eqnarray}
&&\mu({\mathrm{{\rm div}}\textbf{u}_{h}^{(n+1)m}},1)
-d(\xi_{h}^{(n+1)m},1)=(\mathbf{f},x)+\langle\mathbf{f}_{1}, x\rangle,\label{eq-37}\\
&&\kappa_{3}(\xi_{h}^{(n+1)m},1)+({\mathrm{{\rm div}}\textbf{u}_{h}^{(n+1)m}},1)
=\kappa_{1}C_{\eta}(t_{(n+\theta)m}).\label{eq-38}
\end{eqnarray}

Using (\ref{eq-37}) and (\ref{eq-38}), we obtain
\begin{equation}\label{eq-39}
(d+\mu \kappa_{3})(\xi_{h}^{(n+1)m},1)=\mu \kappa_{1}C_{\eta}(t_{(n+\theta)m})-(\mathbf{f}, x)-\langle\mathbf{f}_{1}, x\rangle.
\end{equation}
From the definition of\ $C_{\xi}(t)$ and (\ref{eq-39}), we conclude that (\ref{eq-34}) holds for all $n\geq {1-\theta}$.

Using (\ref{eq-33}), (\ref{eq-34}), (\ref{eq-38}) and the Gauss divergence theorem, we see that (\ref{eq-35}) holds. The proof is complete.
\end{proof}

Using Lemma \ref{lma1912-1} and Lemma \ref{lma1912-2}, taking the similar argument to one of \cite{X.B.Feng2014} or \cite{Evans98}, we can get the following result and omit the detail of its proof here.

\begin{thrm}\label{th1912-1}
Let ${\bf u}_0\in{\bf H}^1(\Omega), {\bf f}\in{\bf L}^2(\Omega), {\bf f}_1\in {\bf L}^2(\partial\Omega), p_0\in L^2(\Omega), \phi\in L^2(\Omega)$, and $\phi_1\in L^2(\partial\Omega)$.
Then there exists a unique numerical solution to the problem (\ref{eq-28})-(\ref{eq-31}).
\end{thrm}



\subsection{Error estimates}

Next, we give the error estimate of the multirate iterative scheme. To do that, we firstly introduce some projection operators. Firstly, for any $\mathbf{v}\in {\bf H}^{1}_\bot(\Omega)$, we define its elliptic projection $\mathcal{R}_{h}: {\bf H}^{1}_\bot(\Omega)\rightarrow {\bf V}_{h}$ by
\begin{equation}\label{eq-47}
(\varepsilon(\mathcal{R}_{h} \mathbf{v}-\mathbf{v}),\varepsilon(\mathbf{w}_{h}))=0\quad \forall\mathbf{w}_{h}\in V_{h}.
\end{equation}
Secondly, for any $\varphi \in H^{1}(\Omega)$ we define the projection operator $\mathcal{S}_{h}: H^{1}(\Omega)\rightarrow M_{h}$ by
\begin{eqnarray}
(\nabla \mathcal{S}_{h} \varphi, \nabla \varphi_{h})=(\nabla \varphi,\nabla \varphi_{h})\quad\forall\mathbf{\psi}_{h}\in M_{h},\label{eq-47-11}\\
(\mathcal{S}_{h} \varphi, 1)=(\varphi,1).\nonumber
\end{eqnarray}
Finally, for any\ $\varphi\in L^{2}(\Omega)$, we define the $L^{2} $-projection $\mathcal{Q}_{h}: L^{2}(\Omega)\rightarrow M_{h}$ as
\begin{equation}\label{eq-50}
(\mathcal{Q}_{h} \varphi-\varphi, \psi_{h})_{K}=0, \forall\psi_{h}\in P_{r_{2}}(K).
\end{equation}

From \cite{S.C.Brenner2008}, we know that the following estimates hold:
\begin{lemm}
 The projection operators of $\mathcal{R}_{h}, \mathcal{S}_{h}, \mathcal{Q}_{h}$ satisfy
\begin{eqnarray}
\|\mathcal{R}_{h}\mathbf{v}-\mathbf{v}\|_{L^{2}(\Omega)}+h\|\nabla(\mathcal{R}_{h}\mathbf{v}-\mathbf{v})\|_{L^{2}(\Omega)}\leq Ch^{s+1}\|\mathbf{v}\|_{H^{s+1}(\Omega)}, 0\leq s\leq k;\label{eq-48}\\
\|\mathcal{S}_{h}\varphi-\varphi\|_{L^{2}(\Omega)}+h\|\nabla(\mathcal{S}_{h}\varphi-\varphi)\|_{L^{2}(\Omega)}\leq Ch^{s+1}\|\varphi\|_{H^{s+1}(\Omega)}, 0\leq s\leq k；\label{eq-49}\\
\|\mathcal{Q}_{h}\varphi -\varphi\|_{L^{2}(\Omega)}+h\|\nabla(\mathcal{Q}_{h}\varphi -\varphi)\|_{L^{2}(\Omega)}\leq Ch^{s+1} \|\varphi\|_{H^{s+1}(\Omega)}, 0\leq s\leq k,\label{eq-51}
\end{eqnarray}
where $k$ is the degree of piecewise polynomial of finite element space.
\end{lemm}

To derive the error estimates for the numerical solution, we split the errors into two parts by the following forms:
\begin{eqnarray}
&&E_{\mathbf{u}}^{n}:~=\mathbf{u}(t_{n})-\mathbf{u}_{h}^{n}=\mathbf{u}(t_{n})-\mathcal{R}_{h}(\mathbf{u}(t_{n}))+\mathcal{R}_{h}(\mathbf{u}(t_{n}))-\mathbf{u}_{h}^{n}:~=\Lambda_{u}^{n}+\Theta_{u}^{n},\nonumber\\
&&E_{\xi}^{n}:~=\xi(t_{n})-\xi_{h}^{n}=\xi(t_{n})-\mathcal{Q}_{h}(\xi(t_{n}))+\mathcal{Q}_{h}(\xi(t_{n}))-\xi_{h}^{n}:~=\Lambda_{\xi}^{n}+\Theta_{\xi}^{n},\nonumber\\
&&E_{\eta}^{n}:~=\eta(t_{n})-\eta_{h}^{n}=\eta(t_{n})-\mathcal{Q}_{h}(\eta(t_{n}))+\mathcal{Q}_{h}(\eta(t_{n}))-\eta_{h}^{n}:~=\Lambda_{\eta}^{n}+\Theta_{\eta}^{n},\nonumber\\
&&E_{p}^{n}:~=p(t_{n})-p_{h}^{n}=p(t_{n})-\mathcal{Q}_h(p(t_{n}))+\mathcal{Q}_h(p(t_{n}))-p_{h}^{n}:~=\Lambda_{p}^{n}+\Theta_{p}^{n}.\nonumber\\
&&E_{p}^{n}:~=p(t_{n})-p_{h}^{n}=p(t_{n})-\mathcal{S}_h(p(t_{n}))+\mathcal{S}_h(p(t_{n}))-p_{h}^{n}:~=\Psi_{p}^{n}+\Phi_{p}^{n}.\nonumber
\end{eqnarray}
Trivially, we have $ \Phi_{p}^{n}=\Lambda_{p}^{n}-\Psi_{p}^{n}+\Theta_{p}^{n} $.

To convenience, we introduce the notations as follows:
\begin{eqnarray}
&&\zeta_{h}^{l+1}=\frac{1}{2}[\mu\|\varepsilon(\Theta_{u}^{(l+1)m})\|_{L^{2}(\Omega)}^{2} + k_{2}\|\Theta_{\eta}^{(l+\theta)m}\|_{L^{2}(\Omega)}^{2} +k_{3}\|\Theta_{\xi}^{(l+1)m}\|_{L^{2}(\Omega)}^{2}],\label{eq200121-30}\\
&&R_{h}^{(n+1)m}=-\frac{1}{m\Delta t} \int_{t_{nm}}^{t_{(n+1)m}}(s-t_{nm})\eta_{tt}(s)ds.\label{eq200121-31}
\end{eqnarray}

\begin{lemm}\label{lem200121-1}
Let 
$\{({\textbf{u}_{h}^{(n+1)m}}, {\xi_{h}^{(n+1)m}}, {\eta_{h}^{(n+\theta)m}})\}$ be the solution of the multirate iterative scheme (\ref{eq-28})-(\ref{eq-31}), 
then we have
\begin{eqnarray}
&&\zeta_{h}^{l+1} +m\Delta t\sum\limits_{n=0}^{l}[\frac{K}{\mu_{f}}(\nabla \Phi_{p}^{(n+1)m}-K\rho_{f}\mathbf{g}, \nabla \Phi_{p}^{(n+1)m})
+\frac{\kappa_{2}\Delta t}{2}\|d_{t}^{(m)} \Theta_{\eta}^{(n+\theta)m}\|_{L^{2}(\Omega)}^{2}]\nonumber\\
&&+m\Delta t\sum_{n=0}^{l}\frac{m\Delta t}{2}(\mu\|d_{t}^{(m)} \varepsilon(\Theta_{u}^{(n+1)m})\|_{L^{2}(\Omega)}^{2}+\kappa_{3} \|d_{t}^{(m)} \Theta_{\xi}^{(n+1)m}\|_{L^{2}(\Omega)}^{2} )   \nonumber\\
&&=\zeta_{h}^{0}+E_{1}+E_{2}+E_{3}+E_{4}+E_{5}+E_{6}+E_{7},\label{eq-52}
\end{eqnarray}
where
\begin{eqnarray}
&&E_{1}=m\Delta t\sum\limits_{n=0}^{l}[(\Lambda_{\xi}^{(n+1)m},\mathrm{{\rm div}}(d_{t}^{(m)} \Theta_{u}^{(n+1)m}))-(\mathrm{{\rm div}}(d_{t}^{(m)} \Lambda_{u}^{(n+1)m}),\Theta_{\xi}^{(n+1)m})],\nonumber\\
&&E_{2}=-k_{3}m\Delta t\sum\limits_{n=0}^{l}(d_{t}^{(m)}\Lambda_{\xi}^{(n+1)m}, \Theta_{\xi}^{(n+1)m}),\nonumber\\
&&E_{3}=m\Delta t\sum\limits_{n=0}^{l}(R_{h}^{(n+\theta)m}, \Phi_{p}^{(n+1)m}),\nonumber\\
&&E_{4}=(1-\theta)k_{1} (m\Delta t)^{2}\sum\limits_{n=0}^{l}(d_{t}^{2(m)} \eta(t_{(n+1)m}),\Theta_{\xi}^{(n+1)m}),\nonumber\\
&&E_{5}=m\Delta t\sum\limits_{n=0}^{l}(d_{t}^{(m)}\Theta_{\eta}^{(n+\theta)m} ,\Psi_{p}^{(n+1)m}-\Lambda_{p}^{(n+1)m}) . \nonumber\\
&&E_{6}=(1-\theta)(m\Delta t)^{2}\sum\limits_{n=0}^{l}\frac{K k_{1}}{\mu_{f}}(d_{t}^{(m)} \nabla \Lambda_{\xi}^{(n+1)m}, \nabla \Phi_{p}^{(n+1)m}).\nonumber\\
&&E_{7}=(1-\theta)(m\Delta t)^{2}\sum\limits_{n=0}^{l}\frac{K k_{1}}{\mu_{f}}(d_{t}^{(m)} \nabla \Theta_{\xi}^{(n+1)m}, \nabla \Phi_{p}^{(n+1)m}).\nonumber
\end{eqnarray}
\end{lemm}

\begin{proof}
Using (\ref{eq-28})-(\ref{eq-30}) and  the definition of $\mathcal{Q}_{h}, ~\mathcal{R}_{h}, ~\mathcal{S}_{h}$, we get
\begin{eqnarray}
&&\mu({\varepsilon(\Theta_{u}^{(n+1)m})},{\varepsilon(\textbf{v}_{h})}) -({\Lambda_{\xi}^{(n+1)m}+\Theta_{\xi}^{(n+1)m}}, {\mathrm{{\rm div}}\textbf{v}_{h}})= 0,\label{eq-53}\\
&&k_{3}(\Lambda_{\xi}^{(n+1)m}+\Theta_{\xi}^{(n+1)m}, \varphi_{h})+ (\mathrm{{\rm div}}\Lambda_{u}^{(n+1)m}+\mathrm{{\rm div}}\Theta_{u}^{(n+1)m}, \varphi_{h}) \nonumber\\ &&=k_{1}({\Lambda_{\eta}^{(n+\theta)m}+\Theta_{\eta}^{(n+\theta)m}}, {\varphi_{h}})+(1-\theta)k_{1}m\Delta t(d_{t}^{(m)} \eta(t_{(n+1)m}),\varphi_{h}), \forall \varphi_{h}\in M_{h}, \label{eq-54}\\
&&(d_{t} \Lambda_{\eta}^{(n+\theta)m}+d_{t}\Theta_{\eta}^{(n+\theta)m},\psi_{h})+\frac{K}{\mu_{f}}(\nabla \Phi_{p}^{(n+1)m}-\rho_{f}\mathbf{g},\nabla \psi_{h})\nonumber\\
&&~~-(1-\theta)m\Delta t\frac{K k_{1}}{\mu_{f}}(d_{t}^{(m)} \nabla E_{\xi}^{(n+1)m}, \nabla \psi_{h})=(R_{h}^{(n+\theta)m},\psi_{h}), ~~~~\forall \psi_{h}\in M_{h}.\label{eq-55}
\end{eqnarray}

Setting $\mathbf{v}_{h}=d_{t}^{(m)} \Theta_{u}^{(n+1)m}$ in (\ref{eq-53}), $\varphi_{h}=\Theta_{\xi}^{(n+1)m}$ after applying the difference operator $d_{t}^{(m)}$ to the equation in (\ref{eq-54}), applying the summation operator $m\Delta t\sum\limits_{n=0}^{l}$ to both sides, and taking $\psi_{h}=\Phi_{p}^{(n+1)m}
=\Lambda_{p}^{(n+1)m}-\Psi_{p}^{(n+1)m}+\kappa_1\Theta_{\xi}^{(n+1)m}+\kappa_2\Theta_{\eta}^{(n+\theta)m}$ in (\ref{eq-55}), after applying the summation operator $m\Delta t\sum\limits_{n=0}^{l}$ to both sides and adding the resulting equations, we see
that (\ref{eq-52}) holds. The proof is complete.
\end{proof}

\begin{thrm}\label{th1912-10}
Suppose that  $\mathbf{u}\in L^{\infty}(0,T;{\bf H}^{1}_{\bot}(\Omega))$, $\xi\in L^{\infty}(0,T; L^{2}(\Omega))$, $\eta \in L^{\infty}(0,T; L^{2}(\Omega))\cap H^{1}(0,T; H^{1}(\Omega)')$, $p\in L^{\infty}(0,T; L^{2}(\Omega))\cap L^{2}(0,T;H^{1}(\Omega))$ are the solution of the problem (\ref{eq-15})-(\ref{eq-20}), and let $\{({\textbf{u}_{h}^{(n+1)m}}, {\xi_{h}^{(n+1)m}}, {\eta_{h}^{(n+\theta)m}})\}_{n\geq0, \theta=0,1}$ be the solution of the multirate iterative scheme (\ref{eq-28})-(\ref{eq-31}),  then we have
\begin{eqnarray}\label{eq-56}
&&\max\limits_{0\leq n\leq l}[\sqrt{\mu}\|\varepsilon(\Theta_{u}^{(n+1)m})\|_{L^{2}(\Omega)} + \sqrt{k_{2}}\|\Theta_{\eta}^{(n+\theta)m}\|_{L^{2}(\Omega)} + \sqrt{k_{3}}\|\Theta_{\xi}^{n+1}\|_{L^{2}(\Omega)}]\nonumber\\
&&+[m\Delta t\sum\limits_{n=0}^{l} \frac{K}{\mu_f}\|\nabla\Phi_{p}^{(n+1)m}\|_{L^{2}(\Omega)}^{2}]^{1/2} \leq C_{1}(T)m\Delta t+C_{2}(T)m h^{2}
\end{eqnarray}
provided that $\Delta t=O(h^{2})$ when $\theta=0$ and $\Delta t>0$ when $\theta=1$, where\ $C_{1}(T)=C\|\eta_{t}\|_{L^{2}((0,T); L^{2}(\Omega))}+C\|\eta_{tt}\|_{L^{2}((0,T);H^{-1}(\Omega))}$~and
$C_{2}(T)=C\|\xi_{t}\|_{L^2((0,T);H^{2}(\Omega))} +C\|\xi\|_{L^{\infty}((0,T);H^{2}(\Omega))}+C\|\mathrm{{\rm div}}\mathbf{u}_{t}\|_{L^{2}((0,T);{\bf H}^{2}(\Omega))}.$
\end{thrm}

\begin{proof} 
Using Lemma \ref{lem200121-1} and the fact of $\Theta_{u}^{0}=0, ~\Theta_{\xi}^{0}=0$,
we have
\begin{eqnarray}
&&\zeta_h^{l+1} +m\Delta t\sum\limits_{n=0}^l[\frac{K}{\mu_f}(\nabla \Phi_{p}^{(n+1)m}-\rho_{f}\mathbf{g}, \nabla \Phi_{p}^{(n+1)m}) +\frac{\kappa_{2}\Delta t}{2}\|d_t \Theta_\eta^{(n+\theta)m}\|_{L^{2}(\Omega)}^2]\nonumber\\
&&+m\Delta t\sum_{n=0}^l\frac{m\Delta t}{2}(\mu\|d_t^{(m)} \varepsilon(\Theta_u^{(n+1)m})\|_{L^{2}(\Omega)}^2+\kappa_3 \|d_t^{(m)} \Theta_\xi^{(n+1)m}\|_{L^{2}(\Omega)}^2 )   \nonumber\\
&&=E_1+E_2+E_3+E_4+E_5+E_6+E_7. \label{eq-57}
\end{eqnarray}

Next, we estimate each term on the right-hand side of (\ref{eq-57}). To bound $E_{1}$, we recall the Korn's inequality:
\begin{eqnarray}
\|\mathrm{{\rm div}} \Theta_u^{(n+1)m}\|_{L^{2}(\Omega)}^2\leq C\|\varepsilon(\Theta_u^{(n+1)m})\|_{L^{2}(\Omega)}^2.\label{eq200122-1}
\end{eqnarray}

Using the Cauchy-Schwarz inequality and (\ref{eq200122-1}), we obtain
\begin{eqnarray}
&&E_{1}
\leq\frac{1}{2}\|\Lambda_\xi^{(l+1)m}\|_{L^{2}(\Omega)}^2+\frac{1}{2}\|{\rm{{\rm div}}} \Theta_u^{(l+1)m}\|_{L^{2}(\Omega)}^2+\frac{1}{2}m\Delta t\sum\limits_{n=1}^l[\|d_t \Lambda_\xi^{(n+1)m}\|_{L^{2}(\Omega)}^2 \nonumber\\
&&+C\|\varepsilon(\Theta_u^{nm})\|_{L^{2}(\Omega)}^2 +\|{\rm div}d_t^{(m)} \Lambda_u^{(n+1)m}\|_{L^{2}(\Omega)}^2+\|\Theta_\xi^{(n+1)m}\|_{L^{2}(\Omega)}^2].\label{eq-58}
\end{eqnarray}

Using the Cauchy-Schwarz inequality and Young inequality, we get
\begin{eqnarray}
E_{2}\leq \frac{k_3}{2}m\Delta t\sum\limits_{n=0}^l(\|d_t^{(m)} \Lambda_\xi^{(n+1)m}\|_{L^{2}(\Omega)}^2+\|\Theta_\xi^{(n+1)m}\|_{L^{2}(\Omega)}^2).\label{eq-59}
\end{eqnarray}


Using the fact that
$$\|R_h^{(n+\theta)m}\|_{H^{-1}(\Omega)}^2\leq\frac{m\Delta  t}{3}\int_{t_{nm}}^{t_{(n+\theta)m}}\|\eta_{tt}\|_{H^{-1}(\Omega)}dt,$$

we can bound  $E_{3}$ as follows:
\begin{eqnarray}\label{eq-61}
&&E_{3}\leq m\Delta t\sum\limits_{n=0}^l\|R_h^{(n+\theta)m}
\|_{H^{-1}(\Omega)}\|\nabla \Phi_p^{(n+1)m}\|_{L^{2}(\Omega)} \nonumber\\
&&\leq m\Delta t\sum\limits_{n=0}^l(\frac{K}{4\mu_f}\|\nabla \Phi_p^{(n+1)m}\|_{L^{2}(\Omega)}^2+\frac{\mu_f m\Delta t}{3K_{1}}\|\eta_{tt}\|_{L^2((t_{nm},t_{(n+1)m});H^{-1}(\Omega))}^2).
\end{eqnarray}

When~$\theta=0$, using the summation by parts formula and $d_t^{(m)}\eta_h(t_0)=0$ to bound $E_{4}$, we have
\begin{eqnarray}
E_{4}=k_1m\Delta t(d_t^{(m)} \eta(t_{(l+1)m}), \Theta_\xi^{(l+1)m})-k_1(m\Delta t)^2\sum\limits_{n=1}^l(d_t^{(m)} \eta(t_{nm}), d_t^{(m)} \Theta_\xi^{(n+1)m}).\label{eq-62}
\end{eqnarray}

Using (\ref{eq-62}), the Cauchy-Schwarz inequality and the Young inequality, we get
\begin{eqnarray}
&&\kappa_{1}m\Delta t(d_t^{(m)} \eta(t_{(l+1)m}),\Theta_\xi^{(l+1)m})\leq \kappa_{1}m\Delta t\|d_t^{(m)}\eta(t_{(l+1)m})\|_{L^{2}(\Omega)}\|\Theta_\xi^{(l+1)m}\|_{L^{2}(\Omega)} \nonumber\\
&&\leq\frac{C\kappa_{1}\mu}{\beta_{1}^2}m(\Delta t)^2 \|\eta_t\|_{L^2((t_{lm},t_{(l+1)m});L^2(\Omega))}^2 +\frac{C\kappa_{1}}{\beta_{1}^2}m(\Delta t)^2\|\Lambda_\xi^{(l+1)m}\|_{L^2(\Omega)}^2\nonumber\\
&&+\frac{\kappa_{1}m\mu}{4}\| \varepsilon(\Theta_{u}^{(l+1)m})\|_{L^2(\Omega)}^2,\label{eq-63}\\
&&\sum\limits_{n=1}^l(d_t^{(m)} \eta(t_{nm}), d_t^{(m)} \Theta_\xi^{(n+1)m})\leq\sum\limits_{n=1}^l\|d_t^{(m)} \eta(t_{nm})\|_{L^{2}(\Omega)}\|d_t^{(m)} \Theta_\xi^{(n+1)m}\|_{L^{2}(\Omega)} \nonumber\\
&&\leq \sum\limits_{n=1}^{l}
[\frac{\mu}{4}\|d_t^{(m)}\varepsilon(\Theta_{u}^{(n+1)m})\|_{L^{2}(\Omega)}^2+\frac{C}{\beta_{1}^{2}} \|d_t^{(m)} \Lambda_{\xi}^{(n+1)m}\|_{L^{2}(\Omega)}^2]\nonumber\\
&&+\frac{C\mu}{\beta_{1}^{2}}\|\eta_t\|_{L^2((0,T);L^2(\Omega))}^2.\label{eq-64}
\end{eqnarray}

We can bound  $E_{5}$ as follows:
\begin{eqnarray}\label{51201}
&&E_{5}= m\Delta t\sum_{n=0}^l(d_t^{(m)}\Theta_\eta^{(n+\theta)m},\Psi_{p}^{(n+1)m}-\Lambda_{p}^{(n+1)m})\\
&&~~~~\leq\frac{m}{2}\Delta t\sum_{n=1}^l\sum\limits_{k=1}^m\| d_t^{(m)}\Theta_\eta^{(n+\theta)m}\|_{L^2(\Omega)}^2 +\frac{m}{2}\Delta t\sum_{n=0}^l(\|\Psi_{p}^{(n+1)m} \|_{L^2(\Omega)}^2+\| \Lambda_{p}^{(n+1)m}\|_{L^2(\Omega)}^2).\nonumber
\end{eqnarray}

We can bound  $E_{6}$ as follows:
\begin{eqnarray}
&&~~~~~~~~~~~~~~~E_{6}=(m\Delta t)^{2}\sum\limits_{n=0}^{l}\frac{K k_{1}}{\mu_{f}}(d_{t}^{(m)} \nabla \Lambda_{\xi}^{(n+1)m}, \nabla \Phi_{p}^{(n+1)m})\\
&&~~~~~~\leq\frac{1}{2}(m\Delta t)^2 \sum_{n=0}^l \frac{K\kappa_1 }{\mu_f}\| d_{t}^{(m)} \nabla \Lambda_{\xi}^{(n+1)m}\|_{L^2(\Omega)}^2+\frac{1}{2}(m\Delta t)^2 \sum_{n=0}^l \frac{K\kappa_1 }{\mu_f}(\| \nabla \Phi_{p}^{(n+1)m} \|_{L^2(\Omega)}^2.\nonumber
\end{eqnarray}

When~$\theta=0$, we can bound  $E_{7}$ as follows:
\begin{eqnarray}
&&\sum\limits_{n=0}^l(d_t^{(m)}\nabla\Theta_\xi^{(n+1)m}, \nabla \Phi_p^{(n+1)m})\nonumber\\
&&\leq
ch^{-1}\sum\limits_{n=0}^l
\|d_t^{(m)}\Theta_\xi^{(n+1)m}\|_{L^{2}(\Omega)}\|\nabla \Phi_p^{(n+1)m}\|_{L^{2}(\Omega)} \nonumber\\
&&\leq\sum\limits_{n=0}^l(\frac{\mu^{2}\kappa_{1}\Delta t}{\beta_1^{2} h^2 }\|d_t^{(m)}\varepsilon(\Theta_{u}^{(n+1)m})\|_{L^{2}(\Omega)}^2+\frac{\kappa_1\Delta t}{h^2\beta_{1}^2}\|d_t^{(m)}\Lambda_{\xi}^{(n+1)m}\|_{L^{2}(\Omega)}^2)\nonumber\\
&&+\sum\limits_{n=0}^l\frac{1}{4\Delta t \kappa_{1}}\|\nabla\Phi_p^{(n+1)m}\|_{L^{2}(\Omega)}^2. \label{eq-65}
\end{eqnarray}

Substituting (\ref{eq-56})-(\ref{eq-65}) into (\ref{eq-57}), we obtain
\begin{eqnarray}
&&\mu\|\varepsilon(\Theta_u^{(l+1)m})\|_{L^{2}(\Omega)}^2+\kappa_2\|\Theta_\eta^{(l+\theta)m}\|_{L^{2}(\Omega)}^2+\kappa_3\|\Theta_\xi^{(l+1)m}\|_{L^{2}(\Omega)}^2\nonumber\\
&&\qquad+ m\Delta t\sum\limits_{n=0}^l\frac{K}{\mu_f}\|\nabla \Phi_p^{(n+1)m}\|_{L^{2}(\Omega)}^2\nonumber\\
&&\leq m(\frac{1}{2}+\frac{C\kappa_{1}}{\beta_{1}^{2}})\| \Lambda_\xi^{(l+1)m}\|_{L^{2}(\Omega)}^2 +\frac{m^{2}}{\mu}\Delta t\sum\limits_{n=0}^l\|d_t^{(m)} \Lambda_\xi^{(n+1)m}\|_{L^{2}(\Omega)}^2\nonumber\\
&&+\frac{1}{2}m \Delta t\sum\limits_{n=0}^l\|\mathrm{{\rm div} }d_t^{(m)} \Lambda_u^{(n+1)m}\|_{L^{2}(\Omega)}^2+(1-\theta)\frac{m^2\mu^{2}\kappa_{1}^{2}K\Delta t}{\mu_{f}\beta_{1}^{2}h^{2}}\|\varepsilon(\Theta_u^{(l+1)m})\|_{L^{2}(\Omega)}^2 \nonumber\\
&&+(\Delta t)^2(\frac{\mu m^{2}}{\beta_{1}}+\frac{\kappa_{1}\mu m}{\beta_{1}^{2}})\|\eta_{t}\|_{L^2((0,T);L^{2}(\Omega))}^2
+\frac{\mu_{f}m\Delta t^{2}}{3K_{1}}\|\eta_{tt}\|_{L^2((0,T);H^{1}(\Omega)')}^{2}\nonumber\\
&&+\frac{m}{2}\Delta t\sum_{n=1}^l\|\Psi_{p}^{(n+1)m} \|_{L^2(\Omega)}^2+\frac{m}{2}\Delta t\sum_{n=1}^l\| \Lambda_{p}^{(n+1)m}\|_{L^2(\Omega)}^2\nonumber\\
&&\leq Cm(\Delta t)^2(\|\eta_{t}\|_{L^2((0,T);L^{2}(\Omega))}^2+\|\eta_{tt}\|_{L^2((0,T);H^{1}(\Omega)')}^{2})\nonumber\\
&&+Cmh^{4}[\|\xi\|_{L^{\infty}((0,T);H^{2}(\Omega))}^{2}+\|\xi_{t}\|_{L^2((0,T);H^{2}(\Omega))}^{2}+\|\mathrm{{\rm div}}\mathbf{u}_{t}\|_{L^{2}((0,T);H^{2}(\Omega))}^{2}].\label{eq-66}
\end{eqnarray}
In (\ref{eq-66}), if $\frac{m^{2}\mu\kappa_{1}^{2}K_{1}\Delta t}{\mu_{f}\beta_{1}^{2}h^{2}}<1$, we can bound the term of $\|\varepsilon(\Theta_u^{(l+1)m})\|_{L^{2}(\Omega)}^2$, so we need the condition:
$\Delta t \leq\frac{\mu_{f}\beta_{1}^{2}h^{2}}{\mu \kappa_{1}^{2}Km^{2}}$. Thus, using~(\ref{eq-48}),~(\ref{eq-49})~,~(\ref{eq-51}), and~(\ref{eq-66}), if\ $\theta=0$, $\Delta t \leq\frac{\mu_{f}\beta_{1}^{2}h^{2}}{\mu \kappa_{1}^{2}Km^{2}}$; when\ $\theta=1$, $\Delta t>0$, we see that~(\ref{eq-56})~holds. The proof is finished.
\end{proof}

\begin{thrm}
Suppose that  $(\mathbf{u}, \xi, \eta)$ and $({\textbf{u}_{h}^{(n+1)m}}, {\xi_{h}^{(n+1)m}}, {\eta_{h}^{(n+\theta)m}})$  are the solutions of the problem (\ref{eq-15})-(\ref{eq-20}) and of the problem (\ref{eq-28})-(\ref{eq-31}), respectively, then we have the following error estimates:
\begin{eqnarray}
&& \max\limits_{0\leq n\leq N}[\sqrt{\mu}\|\varepsilon(\mathbf{u}(t_{nm})-\mathbf{u}_h^{nm})\|_{L^{2}(\Omega)} + \sqrt{\kappa_2}\|\eta(t_{nm})-\eta_h^{nm}\|_{L^{2}(\Omega)} \nonumber\\
&&+ \sqrt{\kappa_3}\|\xi(t_{nm})-\xi_h^{nm}\|_{L^{2}(\Omega)}] \leq \widetilde{C}_1(T)m\Delta t+\widetilde{C}_2(T)mh^2, \label{eq-67}\\
&&[\Delta t\sum\limits_{n=0}^l\sum\limits_{k=1}^m \frac{K}{\mu_f}\|\nabla(p(t_{nm+k})-p_h^{nm+k})\|_{L^{2}(\Omega)}^2]^{1/2}\leq \widetilde{C}_1(T)m\Delta t+\widetilde{C}_2(T)m h,\label{eq-68}
\end{eqnarray}
provided that\ $\Delta t=O(h^2)$ when\ $\theta=0$;\ $\triangle t>0$ when\ $\theta=1$. \ $\widetilde{C}_1(T):~=C_1(T),~\widetilde{C}_2(T):~=C_2(T)+\|\eta\|_{L^{\infty}((0,T); H^{2}(\Omega))}+\|\nabla \mathbf{u}\|_{L^{\infty}((0,T);H^{2})}+\|\xi\|_{L^{\infty}((0,T); H^{2}(\Omega))}$.
\end{thrm}

\begin{proof}
The assertions follow easily from  a straightforward use of the triangle inequality on
\[
\mathbf{u}(t_{nm})-\mathbf{u}_h^{nm}=\Lambda_u^{nm}+\Theta_u^{nm}, \quad  \xi(t_{nm})-\xi_h^{nm}=\Lambda_\xi^{nm}+\Theta_\xi^{nm},
\]
\[
\eta(t_{nm})-\eta_h^{nm}=\Lambda_\eta^{nm}+\Theta_\eta^{nm}, \quad  p(t_{nm+k})-p_h^{nm+k}=\Lambda_p^{nm+k}+\Theta_p^{nm+k},
\]
\[
 p(t_{nm+k})-p_h^{nm+k}=\Psi_p^{nm+k}+\Phi_p^{nm+k},
\]
and applying (\ref{eq-49})-(\ref{eq-51}) and Theorem \ref{th1912-10}, we imply that (\ref{eq-67}) and (\ref{eq-68}) hold. The proof is complete.
\end{proof}


\section{ Numerical tests}\label{sec4}
In this section, we will present three two-dimensional numerical experiments to validate theoretical results for the proposed numerical methods, to numerically examine the performances of the approach and methods as well as to compare them with existing methods in the literature on two benchmark problems. The numerical examples show that our approach and numerical methods have a build-in mechanism to prevent the "locking" phenomenon. Also, we denote CR by the shorthand notation of convergence rates.

\textbf{Test 1}. Let \ $\Omega= [0,1]\times[0,1]$,\ $\Gamma_{1} = \{(x,0);~0\leq x\leq1\}$,\ $~\Gamma_{2}= \{(1,y);~0\leq y\leq1 \}$,\ $\Gamma_{3}= \{(x,1);~0\leq x\leq1 \}$,\ $~\Gamma_{4}= \{ (0,y);~0\leq y\leq1 \}$, and $T=1,\Delta t=1e-6$. We consider problem (\ref{eq-6})-(\ref{eq-12}) with the following source functions:
\begin{align*}
\mathbf{f} &=-(\lambda+\mu)t(1,1)^T+\alpha \cos(x+y)e^t(1,1)^T,\\
\phi &=\Bigl(c_0+\frac{2K}{\mu_f} \Bigr)\sin(x+y)e^t+\alpha(x+y),
\end{align*}
and the following boundary and initial conditions:
\begin{eqnarray}
p=\mathrm{sin}(x+y)e^t   &\mathrm{on} & \partial\Omega_T,\nonumber\\
u_1=\frac{1}{2}x^2 t &\mathrm{on}& \Gamma_j\times (0,T),~j=2,~4,\nonumber\\
u_2=\frac{1}{2}y^2t  &\mathrm{on} & \Gamma_j\times(0,T),~j=1,~3, \nonumber\\
\sigma(\mathbf{u})\mathbf{n}-\alpha pI\mathbf{n}=\mathbf{f}_1  &\mathrm{on}& \partial\Omega_T,\nonumber\\
\mathbf{u}(x,0)=\mathbf{0},~ p(x,0)=\mathrm{sin}(x+y) &\mathrm{in}& \Omega,\nonumber
\end{eqnarray}
 where
\begin{align*}
\mathbf{f}_{1}&=\mu (x n_{1},y n_{2})^{T}t+\lambda(x+y)(n_{1},n_{2})^{T} t-\alpha\sin(x+y)(n_{1},n_{2})^{T}e^{t},
\end{align*}
It is easy to check that the exact solution for this problem is
\begin{align*}
\mathbf{u}=\frac{t}{2}(x^{2},y^{2})^{T},\quad p=\sin(x+y)e^t.
\end{align*}

\begin{table}[H]
\centering
\caption{ Physical parameters}\label{table1}
\begin{tabular}{c l c }
\hline
 Parameter   &  \quad Description      &   \quad   Value    \\
\hline
  $\lambda $  &\quad Lam$\acute{e}$ constant      & \quad 1.43e-4\\
  $\mu$       &\quad Lam$\acute{e}$ constant          &  \quad 3.57e-5 \\
  $c_0$       &\quad Constrained specific storage coefficient &  \quad 1e-5 \\
  $\alpha$    &\quad Biot-Willis constant                       &  \quad  0.83 \\
  $K$         &\quad Permeability tensor             &  \quad $(1e-5)I$\\
  $E$         &\quad Young's modulus                          &  \quad  1e-4\\
  $\nu$       &\quad Poisson ratio                             &  \quad  0.4\\
\hline
\end{tabular}
\end{table}

\begin{table}[H]
\begin{center}
\caption{The errors and convergence rates of\ $\mathbf{u}_h^n$ when\ $m=5$}\label{table2}
\begin{tabular}{ccccc}
\hline
$h$& $\|e_u\|_{L^{\infty}(L^2)}$ & CR& $\|e_u\|_{L^{\infty}(H^1)}$& CR\\
\hline
$0.18$ & 0.00106336& & 0.0666679 & \\
$0.09$ &9.00707e-5 &3.5614 & 0.0116539&2.5162 \\
$0.045$ &7.79098e-6&3.5312&  0.00204438&2.5111 \\
$0.0225$ &6.87406e-7& 3.5026&  0.000359823& 2.5063\\
\hline
\end{tabular}
\end{center}
\end{table}

\begin{table}[H]
\begin{center}
\caption{The errors and convergence rates of\ $\mathbf{u}_h^n$ when\ $m=1$ }\label{table3}
\begin{tabular}{ccccc}
\hline
$h$& $\|e_u\|_{L^{\infty}(L^2)}$ & CR & $\|e_u\|_{L^{\infty}(H^1)}$& CR\\
\hline
$0.18$ & 0.00106336& & 0.0666679 & \\
$0.09$ &9.00707e-5&3.5614 & 0.0116539&2.5162 \\
$0.045$ & 7.79083e-6&3.5312&  0.00204438&2.5111 \\
$0.0225$ &6.85585e-7& 3.5064&  0.000359823& 2.5063\\
\hline
\end{tabular}
\end{center}
\end{table}


\begin{figure}[H]
\centering{
\includegraphics[width=8cm]{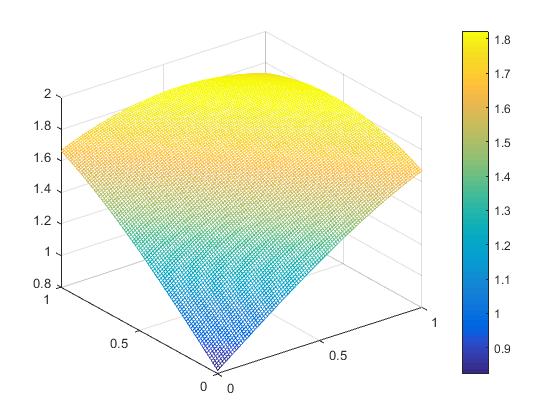}}
\caption{Computed pressure\ $p_{h}^n$ at\ $T=1$ when\ $m=5$.}\label{fig2}
\end{figure}

Table~ \ref{table2} and Table~ \ref{table3} display the errors of displacement $\mathbf{u}_h^n$ in \ $L^{\infty}(0, T; L^2(\Omega))$-norm and $L^{\infty}(0,T; H^1(\Omega))$-norm and the convergence rates with respect to\ $h$ at the terminal time\ $T$ when\ $m=5$ and\ $m=1$, respectively, one can see that the convergence rates is almost identical. However,  the multirate iterative scheme greatly reduces the computational cost, for example, when $m=5,~h=0.18$, the execution time of multirate iterative scheme with multiphysics finite element method is\ $t=6.333s$; when\ $m=1,~h=0.18$, the execution time is\ $t=17.846s$.

If the parameters are same, when\ $m=5,~h=0.18$, the execution time of multirate iterative scheme based on multiphysics discontinuous Galerkin method is\ $t=13.168s$; when\ $m=1,~h=0.18$, the execution time is\ $t=30.302s$. So, we can conclude that the multirate iterative scheme with multiphysics finite element method save a huge computation cost.

Figure \ref{fig2} display the computed pressure \ $p_{h}^n$ at\ $T=1$ when\ $m=5$, from the above two figures, we see that our numerical method has no "locking" phenomenon.

\textbf{Test 2.} In this test, we consider so-called Barry-Mercer's problem, and we set $\Omega=[0,1]\times[0, 1]$ and the boundary segments $\Gamma_{j}, j=1, 2, 3, 4$  are same as  $\textbf{Test 1}$, and $\Delta t=1e-5, T=1$, $\mathbf{f}\equiv0$ and $\phi\equiv0$, and we take the following boundary conditions:
\begin{eqnarray}
p=0  &\mathrm{on}&  \Gamma_j\times(0,T),~j=2,~3,~4,\nonumber\\
p=p_2 &\mathrm{on}&   \Gamma_j\times(0,T),~j=1,\nonumber\\
u_1=0 &\mathrm{on}&   \Gamma_j\times(0,T),~j=2,~4,\nonumber\\
u_2=0 &\mathrm{on}&   \Gamma_j\times(0,T),~j=1,~3,\nonumber\\
\sigma\mathbf{n}-\alpha pI \mathbf{n}=\vec{f}_1:=(0,\alpha p)^T &\mathrm{in}& \partial\Omega_T,\nonumber
\end{eqnarray}
where
\begin{equation}\nonumber
p_2(x,t)=\left\{
\begin{array}{cc}
    \sin t, ~(x, t)\in [0.2, 0.8)\times(0, T),\\
    0,~~~{\rm others}.
\end{array}
    \right.
\end{equation}

\begin{table}[H]
\centering
\caption{Physical parameters}\label{table4}
\begin{tabular}{c l c}
\hline
 Parameter    &  \quad Description     &   \quad  Value    \\
\hline
  $\lambda $  &\quad Lam$\acute{e}$ constant         & \quad 0.0044\\
  $\mu$       &\quad Lam$\acute{e}$ constant         &  \quad 0.0158 \\
  $c_0$       & \quad Constrained specific storage coefficient  &  \quad 0.9 \\
  $\alpha$    &\quad Biot-Willis constant                       &  \quad  0.31 \\
  $K$         &\quad Permeability tensor              &  \quad (3e-6)I\\
  $E$         &\quad Young's modulus                           &  \quad  3.5e-2\\
  $\nu$       &\quad Poisson ratio                           &  \quad  0.11 \\
\hline
\end{tabular}
\end{table}

\begin{table}[H]
\begin{center}
\caption{The errors and convergence rates of\ $\mathbf{u}_h^n$ when\ $m=5$}\label{table5}
\begin{tabular}{ccccc}
\hline
$h$& $\|e_u\|_{L^{\infty}(L^2)}$ & CR& $\|e_u\|_{L^{\infty}(H^1)}$& CR\\
\hline
$0.18$ & 4.7027e-8 & & 4.66882e-7 & \\
$0.09$ & 1.67091e-8 &1.4929 &2.49994e-7& 0.9012\\
$0.045$ & 5.6643e-9 & 1.5607&  1.16283e-7& 1.1043 \\
$0.0225$ &1.44007e-9& 1.9758& 4.8909e-8 & 1.2495\\
\hline
\end{tabular}
\end{center}
\end{table}
\begin{table}[H]
\begin{center}
\caption{The errors and convergence rates of\ $\mathbf{u}_h^n$ when\ $m=1$}\label{table6}
\begin{tabular}{ccccc}
\hline
$h$& $\|e_u\|_{L^{\infty}(L^2)}$ & CR& $\|e_u\|_{L^{\infty}(H^1)}$& CR\\
\hline
$0.18$ & 1.56757e-8 & & 1.55627e-7 & \\
$0.09$ & 5.56969e-9 &1.4929 &8.33314e-8& 0.9012 \\
$0.045$ & 1.8881e-9 & 1.5607&  3.8761e-8& 1.1043 \\
$0.0225$ &4.80022e-10& 1.9758&  1.6303e-8 & 1.2495\\
\hline
\end{tabular}
\end{center}
\end{table}
\begin{figure}[H]
\centering{
\includegraphics[width=8cm]{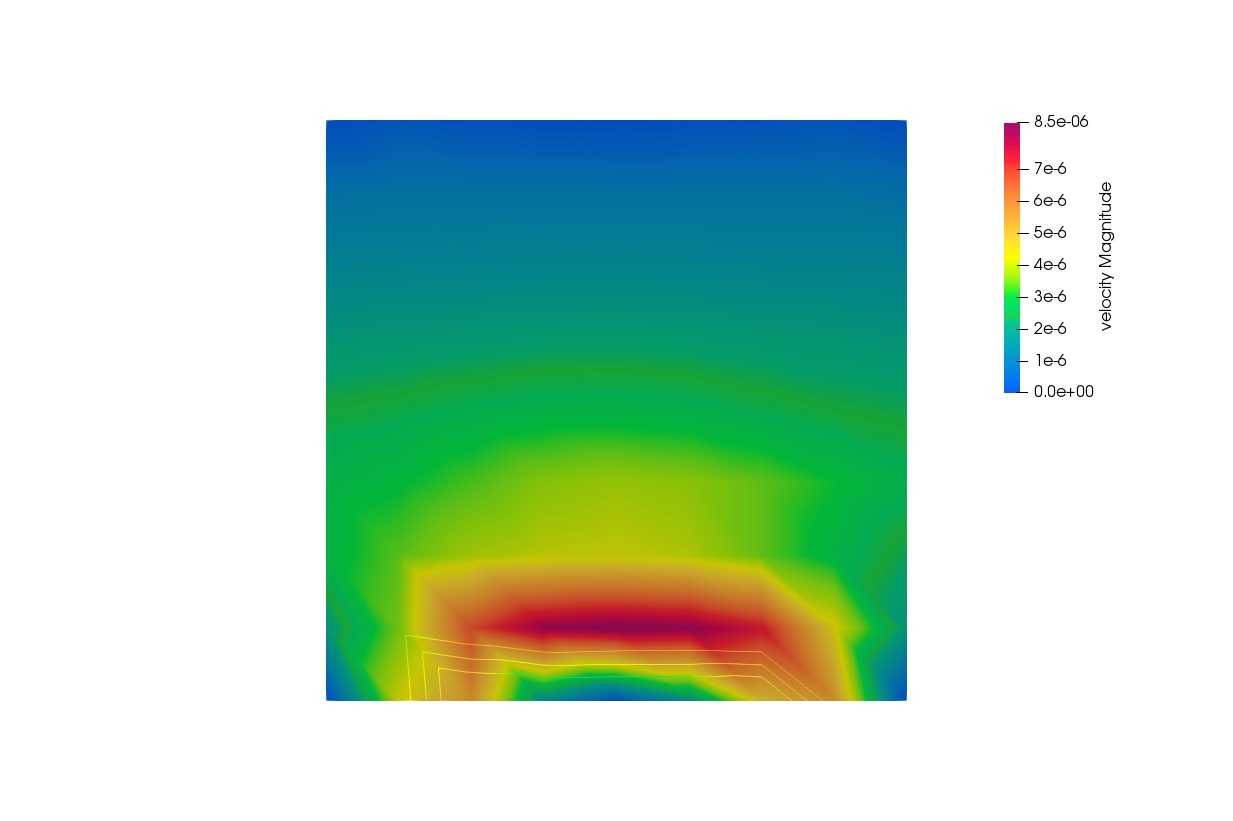}}
\caption{The computed displacement \ $\mathbf{u}_{h}^n$ at\ $T=1$ when\ $m=5$.}\label{fig3}
\end{figure}
\begin{figure}[H]
\centering{
\includegraphics[width=8cm]{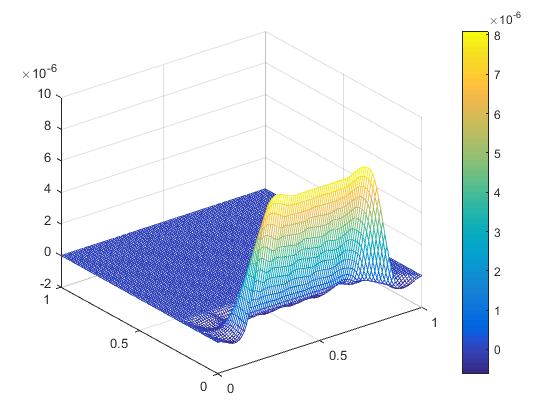}}
\caption {The computed pressure \ $p_{h}^n$ at\ $T=1$ when\ $m=5$.}\label{fig4}
\end{figure}

From Table \ref{table5} and Table \ref{table6}, we know that the errors of displacement \ $\mathbf{u}_h^n$ when \ $m=1$ are better than ones of \ $\mathbf{u}_h^n$ when\ $m=5$, and the convergence rates are identical. However,   the execution times of multirate iterative scheme based on multiphysics finite element method are\ $t=208.977s$ for the case of \ $m=5$ and $t=291.067s$ for the case of \ $m=1$ (when $h=0.18$). Also, if the parameters are same,  the execution time of multirate iterative scheme based on multiphysics discontinuous Galerkin method is\ $t=300.597s$ when\ $m=5,~h=0.18$,  and the execution time is\ $t=372.01s$ when\ $m=1,~h=0.18$. So, we can conclude that the multirate iterative scheme greatly reduces the computational cost.

Figure \ref{fig3}~and Figure \ref{fig4}~display the computed displacement \ $\mathbf{u}_h^n$ and pressure \ $p_{h}^n$ at\ $T=1$ when\ $m=5$, respectively. From the above two figures, we see that there is no "locking" phenomenon.

\textbf{Test 3}. Again, we consider problem with $\Omega=[0,1]\times[0,1]$. Let $\Gamma_{j}$ be same as in \textbf{Test 1}, and $\Delta t=1e-5, T=1$. There is no source, that is, $\mathbf{f}\equiv0$ and $\phi\equiv0$. The boundary conditions are taken as
\begin{eqnarray}
-\frac{K}{\mu_{f}} (\nabla p-\rho_{f} \mathbf{g} )\cdot \mathbf{n}=0    &\mathrm{on} & \partial\Omega_{T},\nonumber\\
\mathbf{u}=\mathbf{0} &\mathrm{on}& \Gamma_{3}\times(0,T), \nonumber\\
\sigma \mathbf{n}-\alpha pI \mathbf{n}= \mathbf{f}_{1} &\mathrm{on}& \Gamma_{j}\times(0,T),~j = 1,~2,~3,\nonumber
\end{eqnarray}
where $\mathbf{f}_{1}= (f_{1}^{1}, f_{1}^{2})^T$ and
\begin{eqnarray}
f_{1}^{1}\equiv0&\mathrm{on}&\partial\Omega_{T},\nonumber
\end{eqnarray}
\qquad\qquad\qquad\begin{equation}
 f_{1}^{2}=\left\{
  \begin{array}{c}
         0, ~~~ (x,t)\in\Gamma_{j}\times(0 , T ), j = 1, 2, 4, \\
         -1,~~~(x,t)\in\Gamma_{3}\times(0 , T ).
       \end{array}
       \right.
\end{equation}
The zero initial conditions are assigned for both $\bf{u}$ and $p$ in this test.

\begin{table}[H]
\centering
\caption{Physical parameters}
\begin{tabular}{c l c }
\hline
 Parameter    &  \quad Description     &   \quad  Value    \\
\hline
  $\lambda $  &\quad Lam$\acute{e}$ constant         & \quad 142857.14\\
  $\mu$       &\quad  Lam$\acute{e}$ constant         &  \quad 35714.29 \\
  $c_0$       & \quad Constrained specific storage coefficient  &  \quad 0.01 \\
  $\alpha$    &\quad Biot-Willis constant                       &  \quad  0.93 \\
  $K$         &\quad Permeability tensor              &  \quad (1e-1)I\\
  $E$         &\quad Young's modulus                           &  \quad  1e5\\
  $\nu$       &\quad Poisson ratio                           &  \quad  0.4 \\
\hline
\end{tabular}\label{table7}
\end{table}

\begin{table}[H]
\begin{center}
\caption{The errors and convergence rates of\ $p_h^n$ when\ $m=5$}\label{table8}
\begin{tabular}{ccccc}
\hline
$h$& $\|e_p\|_{L^{2}(L^2)}$ & CR& $\|e_p\|_{L^{2}(H^1)}$& CR\\
\hline
$0.18$ & 1.11026e-9 & & 3.38447e-8 & \\
$0.09$ & 3.286e-10 &1.7565 & 2.17254e-8& 0.6395 \\
$0.045$ &1.04069e-10 & 1.6588&  1.51649e-8& 0.5186 \\
$0.0225$ &2.38175e-11& 2.1274&  6.54142e-9& 1.2131\\
\hline
\end{tabular}
\end{center}
\end{table}
\begin{table}[H]
\begin{center}
\caption{The errors and convergence rates of\ $p_h^n$ when\ $m=1$}\label{table9}
\begin{tabular}{ccccc}
\hline
$h$& $\|e_p\|_{L^{2}(L^2)}$ & CR& $\|e_p\|_{L^{2}(H^1)}$& CR\\
\hline
$0.18$ & 6.38714e-10 & & 1.94327e-8 & \\
$0.09$ &1.91689e-10 &1.7364 & 9.9166e-9& 0.9706 \\
$0.045$ & 6.1088e-11 & 1.6498&  5.00331e-9& 0.9870 \\
$0.0225$ &1.35621e-11& 2.1713&  3.23977e-9& 0.6270\\
\hline
\end{tabular}
\end{center}
\end{table}

\begin{figure}[H]
\centering{
\includegraphics[width=8cm]{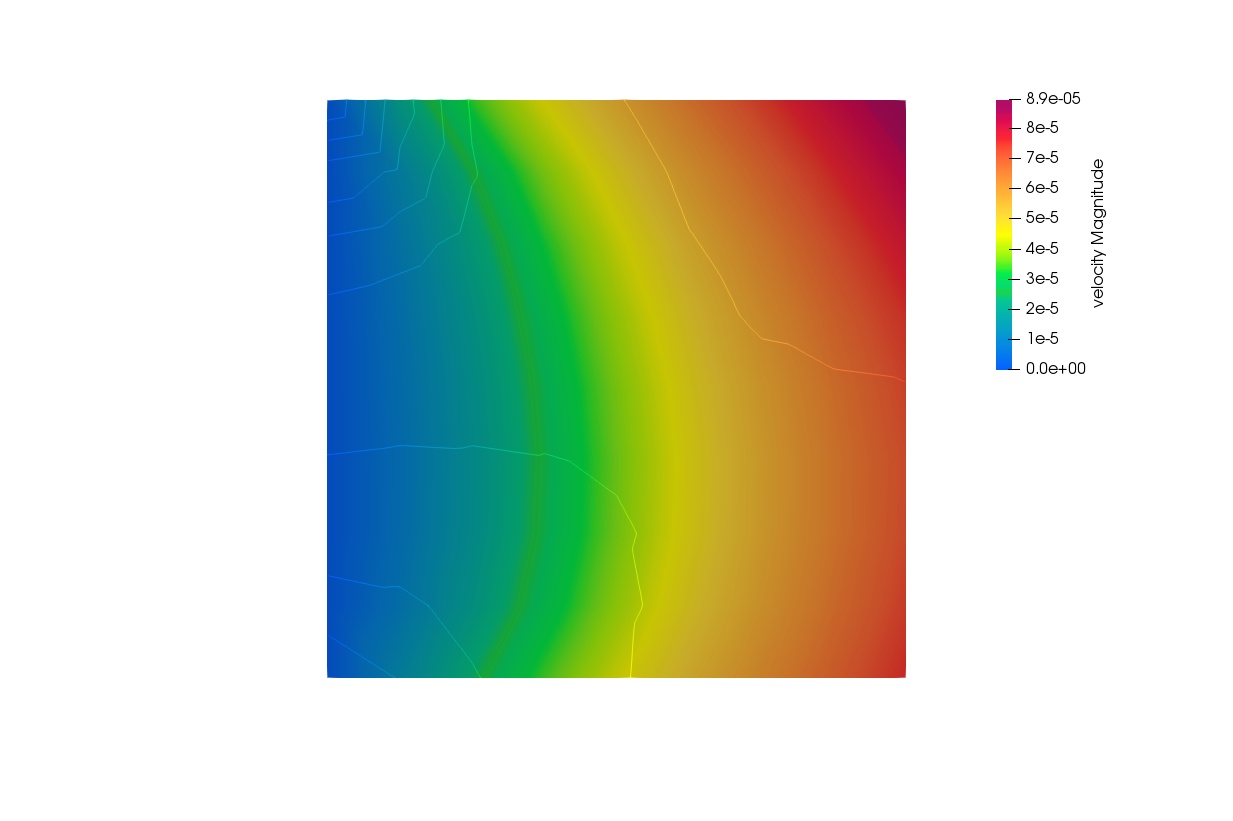}}
\caption{Computed displacement\ $\mathbf{u}_{h}^n$ at\ $T=1$ when\ $m=5$.}\label{fig5}
\end{figure}
\begin{figure}[H]
\centering{
\includegraphics[width=8cm]{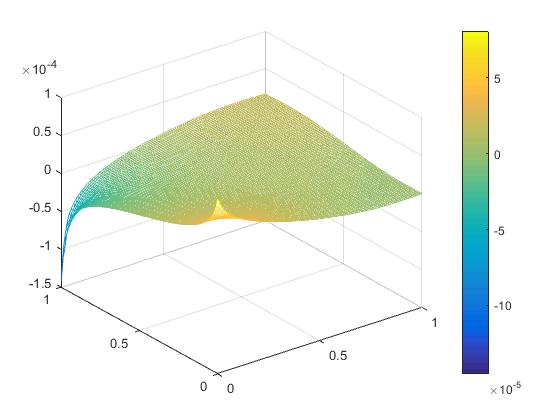}}
\caption{Computed pressure\ $p_{h}^n$ at\ $T=1$ when\ $m=5$.}\label{fig6}
\end{figure}

From Table \ref{table8} ~and Table \ref{table9}, we know that the errors of \ $p_h^n$ in \ $L^{2}(0,T;L^2(\Omega))$-norm and \ $L^{2}(0,T;H^1(\Omega))$-norm when \ $m=1$ are better than ones of $m=5$. However, the convergence rates are almost equal for the cases of \ $m=5$ and\ $m=1$. When\ $m=5,~h=0.18$, the execution time of multirate iterative scheme based on multiphysics finite element method is\ $t=215.692s$; when\ $m=1,~h=0.18$, the execution time is\ $t=282.331s$, if the parameters are same, when\ $m=5,~h=0.18$, the execution time of multirate iterative scheme based on multiphysics discontinuous Galerkin method is\ $t=342.197s$; when\ $m=1,~h=0.18$, the execution time is\ $t=421.611s$. So, we can conclude that the multirate iterative scheme greatly reduces the computational cost. Also, from Figure \ref{fig5}~and Figure \ref{fig6}, we see that the numerical method in this paper has no "locking" phenomenon.

\section{Conclusion}\label{sec5}
In this paper, we propose a multirate iterative scheme with multiphysics finite element method for a poroelasticity model. And we prove that the multirate iterative scheme is stable and the numerical solution satisfies some energy conservation laws, and it doesn't reduce the precision of numerical solution and greatly reduces the computational cost. In the future work, we will apply the proposed approaches to more complex practical problems and nonlinear poroelasticity model.

\end{document}